\documentclass[a4paper,12pt]{amsart}
\usepackage{amsmath,amssymb}
\def\B{\Bbb B}
\def\C{\Bbb C}
\def\D{\Bbb D}
\def\d{\delta}
\def\vp{\varphi}
\def\si{\sigma}
\def\eps{\varepsilon}
\def\eg{g^\ast}
\def\ds{\displaystyle}
\newtheorem{thm}{Theorem}
\newtheorem{proposition}[thm]{Proposition}
\newtheorem{corollary}[thm]{Corollary}

\title{Strong localization of invariant metrics}

\author{John Erik Forn{\ae}ss}
\address{J.E. Forn{\ae}ss\\Department of Mathematics, Norwegian University of
Science and Technology\\Sentralbygg 2, Alfred Getz vei 1, 7491 Trondheim, Norway}
\email{john.fornass@ntnu.no}

\author{Nikolai Nikolov}
\address{N. Nikolov\\Institute of Mathematics and Informatics\\Bulgarian Academy
of Sciences\\Acad. G. Bonchev 8, 1113 Sofia, Bulgaria
\vspace{1mm}
\newline Faculty of Information Sciences\\
State University of Library Studies and Information Technologies\\
Shipchenski prohod 69A, 1574 Sofia,
Bulgaria}

\email{nik@math.bas.bg}

\thanks{The second named author was partially supported by the Bulgarian National Science Fund,
Ministry of Education and Science of Bulgaria under contract DN 12/2. This paper was started
while the first named author was visiting the Institute of Mathematics and Informatics, Bulgarian Academy
of Sciences in September 2019.}


\begin{document}


\begin{abstract} A quantitative version of strong localization of the Kobayashi, Azukawa and Sibony
metrics, as well as of the squeezing function, near a plurisubharmonic peak boundary point of a domain
in $\C^n$ is given. As an applicatio\-n, the behavior of these metrics near a strictly pseudoconvex
boundary point is studied. A weak localization of the three metrics and the squeezing function
is also given near a plurisubharmonic antipeak boundary point.
\end{abstract}

\maketitle

\section{Introduction}

Denote by $\D\subset\C$ the unit disc. Let $D$ be an open set in $\C^n.$
The Ko\-bayashi, Azukawa and Sibony metrics of $D$ at $(z,X)\in D\times\C^n$
are defined in the following way:

$$K_D(z;X)=\inf\{|\alpha| :\exists f\in\mathcal O(\D,D),
\ f(0)=z,\ \alpha f'(0)=X\};$$

$$A_D(z;X)=\limsup_{\lambda\nrightarrow 0}\frac{\eg_D(z,z+\lambda X)}{|\lambda|},$$
where $\eg_D=\exp g_D$ and
$$g_D(z,w)=\sup\{u(w):u \in\mbox{PSH}(D),\ u<0,\ u<\log|\cdot-z|+C\}$$
is the pluricomplex Green function of $D$ with pole at $z;$

$$S_D(z;X)=\sup_v[L_v(z;X)]^{1/2},$$
where $L_v$ is the Levi form of $v$, and the supremum is taken over all log-psh functions $v$ on $D$
such that $0\le v<1,$ $v(z)=0,$ and $v$ is $C^2$ near $z$ (log=logarithm, (p)sh= (pluri)subharmonic).

It is well-known that $S_D\le A_D\le K_D.$

For various properties of these metric we refer the reader to \cite{JP}.

Denote now by $\B_n=\B_n(0,1)$ the unit ball in $\C^n.$ For any holomorphic embedding $f:D\to\B_n$
with $f(z)=0,$ set $$\sigma_D(f,z)=\sup\{r>0:r\B_n\subset f(D)\}.$$
The squeezing function of $D$ is defined by $\ds \sigma_D(z)=\sup_f\sigma_D(f,z)$
if such $f$'s exist, and $\sigma_D(z)=0$ otherwise (that is, if $D$ is not biholomorphic
to a bounded open set) -- see e.g. \cite{NT,NV} and the references therein.

Recall now that a point $p\in\partial D$ is called a psh peak point (resp. antipeak) if there exists
a psh function $\vp$ on $D$ such that $\ds\lim_{z\to p}\vp(z)=0$ and $\ds\sup_{D\setminus U}\vp<0$
(resp.~$\ds\lim_{z\to p}\vp(z)=-\infty$ and $\ds\inf_{D\setminus U}\vp>-\infty$)
for any neigh\-borhood $U$ of $p.$ Note that the notion of psh peak point has a local character,
and such a point admits a negative psh antipeak function (see the proof of \cite[Lemma 2.1.1]{Gau}).
Assuming that $\vp=\log|f|,$ where $f\in\mathcal O(D,\D),$ we define the notion of holomorphic peak point.

Strong localization of invariant metrics as in \eqref{sl} below plays crucial rule in many
boundary problems in complex analysis. Such a localization for $K_D$ near a psh peak point follows by
\cite[Proposition 2.1.a)]{Ber} (see also \cite[Lemma 2.1.1]{Gau}). The same is true for $A_D$ near a
holomorphic peak point (see \cite[Corollaries 1 \& 2]{Nik}). A quantitative version of strong localization
for $K_D$ near special holomorphic peak points is given in \cite[Theorem 2.1 \& Lemma 2.2]{FR}.

The main aim of this note is to give a quantitative version of strong localization for
$M_D\in\{K_D, A_D, S_D\}$ and $\sigma_D$ near a psh peak point in terms of the respective psh peak function.

\begin{thm}\label{thm}
Let $D$ be a domain in $\C^n.$ Suppose that there exists a psh peak function
$\vp$ for $p\in\partial D;$ $\vp$ is assumed $C^2$ near $p$ if $M_D=S_D.$
Then for any bounded neighborhood $U$ of $p$ there are
a neighborhood $V\subset U$ of $p$ and a constant $m>0$ such that for
$z\in D\cap V$ one has that
\begin{equation}\label{in}
M_D(z;X)\ge e^{m\vp(z)}M_{D\cap U}(z;X),\quad X\in\C^n,
\end{equation}
\begin{equation}\label{si}
\sigma_{D\cap U}(z)\ge e^{m\vp(z)}\sigma_D(z).
\end{equation}

In particular,

(i) since $M_{D\cap U}\ge M_D,$ then
\begin{equation}\label{sl}
\lim_{z\to p}\frac{M_D(z;X)}{M_{D\cap U}(z;X)}=1\mbox{ uniformly in }X\in(\C^n)_\ast;
\end{equation}

(ii) since $\sigma_{D\cap U}\le 1,$ then $\ds\lim_{z\to p}\sigma_D(z)=1$ implies
$\ds\lim_{z\to p}\sigma_{D\cap U}(z)=1.$
\smallskip

\noindent $(\ast)$ In addition, if $D$ is bounded, then $V\Subset U$ can be chosen arbitrary.
\end{thm}

\noindent{\it Remark.} Theorem \ref{thm} (ii) is exactly \cite[Proposition 2]{NV}.
The inverse impli- cation cannot be true without global assump\-tions about $D;$ it is
even possible $\si_D=0$ but $\si_{D\cap U}=1.$
\smallskip

When $D$ is an unbounded domain in $\C^n$ and $p=\infty,$ we use the same definition of a
psh peak point as above (see \cite[Definition 1.4. (a)]{Gau}). Then we have the following
counterpart of Theorem \ref{thm}.

\begin{proposition}\label{infty} Let $D$ be an unbounded domain in $\C^n.$
Suppose that there exists a psh peak function $\vp$ for $p=\infty.$ Then for any neighborhood
$U$ of $\infty$ there are a neighborhood $V\subset U$ of $\infty$ and a constant $m>0$ such that
for $z\in D\cap V$ one has that
$$K_D(z;X)\ge e^{m\vp(z)}K_{D\cap U}(z;X),\quad X\in\C^n,$$
$$\sigma_{D\cap U}(z)\ge e^{m\vp(z)}\sigma_D(z).$$
\end{proposition}

Note that if $p=\infty$ is a psh peak point of $D,$ then any bounded subset of
$D$ is uniformly $M_D$-hyperbolic; more precisely:

\begin{proposition}\label{hyper} Let $p=\infty$ be a psh peak point of an unbounded domain
$D$ in $\C^n.$ Then for any $r>0$ there exists a constant $c>0$ such that
$$M_D(z;X)\ge c|X|,\quad z\in D\cap r\B_n,\ X\in\C^n.$$
\end{proposition}

\section{Proof of Theorem \ref{thm}}

{\it The case $M_D=K_D.$} By \cite[Lemma 2]{Roy}, we have that
$$K_D(z;X)\ge K_{D\cap U}(z;X)\inf_{D\setminus U}l_D(z,\cdot),$$
where
$$l_D(z,w)=\inf\{|\alpha|: \exists f\in\mathcal O(\D,D),\ f(0)=z,\ f(\alpha)=w\}.$$
Note that
$$l_D(z,w)=l_D(w,z)\ge\eg_D(w,z)\mbox{ and }\inf_{D\setminus U}l_D(z,\cdot)=\inf_{D\cap\partial U}l_D(z,\cdot)$$
($\ge$ follows by the Schwarz lemma for log-sh functions). Hence
\begin{equation}\label{gr}
K_D(z;X)\ge K_{D\cap U}(z;X)\inf_{D\cap\partial U}\eg_D(\cdot,z).
\end{equation}

Let now $W=\B_n(p,1)$ and $\theta$ be a negative psh antipeak function for $D$ at $p.$
We may assume that $U=\B_n(p,r)$ ($r<1$) and
$$\inf_{D\setminus W}\theta\ge c=1+\sup_{D\cap\overline U}\theta$$
($\inf=c$ if $D\subset W$). Setting
$\tilde\theta=1+(1-r^2)(\theta-c),$ then
$$\hat\theta=\left\{\begin{array}{ll}
|\cdot-p|^2,&D\cap U\\
\max\{|\cdot-p|^2, \tilde\theta\},&D\cap W\setminus U\\
 \tilde\theta,&D\setminus W
\end{array}\right.$$
is a bounded psh function on $D.$

Let $\chi:[0,\infty)\to [0,1]$ be $C^\infty$ such that $\chi=1$ on
$[0,(1-r)/2]$ and $\mbox{supp }\chi\in[0,1-r].$ For any neighborhood $\hat U\Subset U$ of $p$,
we may choose first $\hat m$ and then $m$ such that
$$\psi=\chi(|\cdot-w|)\log|\cdot-w|+\hat m(\hat\theta-\sup_D\hat\theta-1)$$
and
$$\hat\vp=\left\{\begin{array}{ll}
\max\{\psi,m\vp\},&D\cap\hat U\\
\psi,&D\setminus\hat U
\end{array}\right.$$
to be psh functions on $D,$ when $w\in D\cap\partial U.$ Then
$$g_D(w,\cdot)\ge\hat\vp\ge m\vp\mbox{\ \ on\ \ }D\cap\hat U$$
which implies \eqref{in} for $M_D=K_D.$
\smallskip

{\it The case $M_D=A_D.$} Let $\ds a=\sup_{D\setminus U}\vp.$
We may choose a neighborhood $W\subset U$ of $p$ such that
$$0<c=\inf_{D\cap W}\vp-a .$$
Let $V\Subset W$ be a neighborhood of $p$ and
$$m=-c^{-1}\inf\{g_{D\cap U}(z,w):z\in D\cap V,\ w\in D\cap U\setminus W\}.$$
For $z\in D\cap V$ and $w\in D\cap U,$ set $v_z(w)=g_{D\cap U}(z,w)+m\vp(w)$ and
$$u_z=\left\{\begin{array}{ll}
v_z,&D\cap W\\
\max\{v_z,ma\},&D\cap U\setminus W\\
ma,&D\setminus U
\end{array}\right..$$
Then $u_z<0$ is psh on $D$ which implies \eqref{in} for $M_D=A_D.$
\smallskip

{\it The case $M_D=S_D.$} Since $p$ is a psh antipeak point,
an obvious modification in the proof of \cite[Theorem 1]{Nik}
(see also the construction of $\hat\theta$ above)
implies that one may find a ball $W=\B_n(p,r)$ and a constant $c>0$ such that for any
$z\in D\cap W$ there exists a psh function $\theta_z<c$ on $D$ with
$$\theta_z(w)=\log|w-z|,\quad w\in D\cap W.$$

The rest of the proof is similar to that of \cite[Lemma 5]{FL}.
We may assume that $U=\B_n(p,r/5).$ Let $\eps\in(0,r^{-2}/2],$ $z\in U$ and
$v_z$ be a competitor for $S_{D\cap U}(z;X).$
Setting
$$\tilde\theta_z=3\theta_z-3\log r+\log 9,\quad \tilde v_z=v_z+\eps e^{2\theta_z},$$
then
$$\hat v_z=\left\{\begin{array}{ll}
\max\{\tilde\theta_z,\tilde v_z\},&D\cap\B_n(z,r/2)\\
\tilde\theta_z,&D\setminus\B_n(z,r/2)
\end{array}\right..$$
is a psh function and $\hat v_z=\tilde v_z$ near $z.$

If $w\in D\cap U,$ then $|w-z|<2r/5$ and hence $\tilde\theta_z(w)<0.$ So, we may
choose a number $m$ such that $\tilde\theta_z+m\vp<0$ for any $z\in D\cap U.$
Then $(1+\eps e^{2c})^{-1}e^{\hat v_z+m\vp}$ is a competitor for $S_D(z,X)$ which implies that
$$(1+\eps e^{2c})S_D(z;X)e^{m\vp(z)}S_{D\cap U}(z;X).$$
It remains to let $\eps\to 0.$
\smallskip

{\it The case $\sigma_D.$} The proof of \cite[Proposition 2]{NV} implies that
\begin{equation}\label{sq}\sigma_{D\cap U}(z)\ge\sigma_D(z)\inf_{D\setminus U}l_D(z,\cdot)
\ge\sigma_D(z)\inf_{D\cap\partial U}\eg_D(\cdot,z).
\end{equation}
Then \eqref{si} follows as \eqref{in} in the case $M_D=K_D.$
\smallskip

{\it Finally,} to prove $(\ast)$, note that a weak localization for $M_D$ holds near any boundary point
(see e.g. \cite[Lemma 2]{Roy} or \cite[Lemma 3]{FL}, \cite[Remark]{Nik} and \cite[Lemma 5]{FL}
for $M_D=K_D,$ $M_D=A_D$ and $M_D=S_D,$ respectively; see also Proposition \ref{weak}).
Then a compactness argument provides a constant $c>0$ such that
\begin{equation}\label{wl}
M_D\ge c M_{D\cap U}\mbox{ on }(D\cap V)\times\C^n,
\end{equation}
and now $(\ast)$ easily follows.
\smallskip

\noindent{\it Remark.} The proof of the case $M_D=A_D$ implies that if $\vp<0$ is a psh function on a bounded
domain $D$ in $\C^n,$ $K$ and $L$ are disjoint compacts in $\C^n,$ and $\ds\sup_{D\cap\partial K}\vp<0,$
then there exists a constant $m>0$ such that
$$g_D(z,w)\ge m\vp(w),\quad z\in D\cap K, w\in D\cap L.$$
For example, this can be applied to any compact subset $K$ of a bounded hyperconvex domain $D$
with an exhaustion function $\vp$ (that is, $\vp\in\mbox{PSH}(D),$ $\vp<0,$
and $\ds\lim_{z\to\partial D}\vp(z)=0$).

\section{Proofs of Propositions \ref{infty} and \ref{hyper}}

\noindent{\it Proof of Proposition \ref{hyper}.} We will assume that $z,w\in D$ and $|z|\le r.$

Let $\vp$ be a psh peak function at $\infty.$ Choose $s>r+1$ such that
$$\inf_{|w|>s}\vp(w)=:\alpha>\beta:=\sup_{|w|\le r+1}\vp(w).$$
Set $\theta_z(w)=\log|w-z|+\frac{\beta}{\alpha-\beta}\log(s+r),$
$\eta=\frac{\log(s+r)}{\alpha-\beta}\vp,$
$$\psi_z(w)=\left\{\begin{array}{ll}
\theta_z(w),&|w-z|\le 1\\
\max\{\theta_z(w),\eta(w)\},&|w-z|>1,\ |w|<s\\
\eta(w),&|w|\ge s
\end{array}\right..$$
Then we may take $e^{2\psi_z}$ as a candidate in the definition of $S_D(z;X)$ which implies that
$$M_D(z;X)\ge S_D(z;X)\ge (s+r)^{\frac{\beta}{\alpha-\beta}}|X|.\qed$$

\noindent{\it Proof of Proposition \ref{infty}.} We may assume that $U=W_r:=\{z\in\C^n:|z|>r\}.$
Keeping the notations from the previous proof and setting $m=\frac{\log(s+r)}{\alpha-\beta},$
$V=W_s,$ it follows that
$$g_D(w,z)\ge m\vp(z),\quad w\in D\setminus U,\ z\in D\cap V.$$
Then \eqref{gr} and \eqref{sq} complete the proof.\qed

\section{Further results}

The next proposition is known (with a different proof) in the case, when $D$ is bounded and $M_D=K_D$
(see \cite[p.~244, Remark]{FR}).

Set $d_D=\mbox{dist}(\cdot,\partial D),$ $\d_D=-d_D$ on $D,$ and $\d=d_D$ otherwise.

Recall that a point $p\in\partial D$ is said to be strictly pseudoconvex if $\partial D$ is
$C^2$ near $p$ and $L_\d(p;X)>0$ for any $X\in T^\C_p(\partial D),$ $X\neq 0.$

\begin{proposition}\label{prop}
Let $p$ be a strictly pseudoconvex boundary point of a domain $D$ in $\C^n.$
Then for any neighborhood $U$ of $p$ there are a neighborhood $V\subset U$ of $p$
and a constant $c>0$ such that for $z\in D\cap V$ one has that
$$M_D(z;X)\ge (1-cd_D(z))M_{D\cap U}(z;X),\quad X\in\C^n,$$
$$\sigma_{D\cap U}(z)\ge (1-cd_D(z))\sigma_D(z).$$
In addition, if $D$ is bounded, then $V\Subset U$ can be chosen arbitrary.
\end{proposition}

\noindent{\it Remark.} The estimate for the squeezing function is optimal. Indeed,
let $D=\B^n$ and $U$ be a neighborhood of $p\in\partial D$ such that $D\cap U$ is not biholomorphic
to $D.$ Then $\sigma_D=1$ and, by \cite[Theorem 1.2]{DFW},
$\sigma_{D\cap U}\le 1-c'd_D$ near $p$ for some $c'>0.$

\begin{proof} It is well-known that there exist a constant $c'>0,$ a neighborhood $U'$
of $p$ and a continuous function $h$ in the closure of
$(\partial D\cap U')\times(D\cap U')$ such that for any $q\in\partial D\cap U':$

\noindent (i) $h(q;\cdot)$ is a holomorphic peak function for $D\cap U'$ at $q;$

\noindent (ii) $|1-h(q;z)|\le c'd_D(z),$ $z\in D\cap U'\cap n_q,$ where $n_q$ is the inner normal to
$\partial D$ at $q.$

Setting $\tilde\vp=\log|h|,$ it remains to repeat the proof of Theorem \ref{thm} for $q$ near $p.$
\end{proof}

\begin{corollary}\label{c1} Let $p$ be a strictly pseudoconvex point of a domain $D$ in $\C^n.$
Then there exist a neighborhood $V$ of $p$ and a constant $c>0$ such that
$$1\ge\frac{A_D(z;X)}{K_D(z;X)}\ge\frac{S_D(z;X)}{K_D(z;X)}\ge 1-cd_D(z),\quad z\in D\cap V,\ X\in\C^n.$$
\end{corollary}

\begin{proof} There exists a bounded neighborhood $U$ of $p$ such that $D\cap U$ is biholomorphic to a
convex domain. Then Lempert's theorem implies that $K_{D\cap U}=A_{D\cap U}=S_{D\cap U}.$ It remains to apply
Proposition \ref{prop}.
\end{proof}

\begin{corollary}\label{c2} Let $\eps\in (0,1],$ $k\in\{0,1\},$ $\eps_k=\frac{k+\eps}{2},$ and $p$ be
a $C^{k+2,\eps}$-smooth strictly pseudoconvex boundary point of a domain $D$ in $\C^n.$
Then there exist a neighborhood $V$ of $p$ and a constant $c>0$ such that
\begin{multline*}
(1-Cd_D(z)^{\eps_k})\Bigl(\frac{|\langle\partial d_D(z),X\rangle|^2}{d^2_D(z)}+\frac{L(z;X)}{d_D(z)}\Bigr)^{1/2}
\le M_D(z;X) \\
\le (1+Cd_D(z)^{\eps_k})\Bigl(\frac{|\langle\partial d_D(z),X\rangle|^2}{d^2_D(z)}+\frac{L(z;X)}{d_D(z)}\Bigr)^{1/2},
\quad z\in D\cap V,\ X\in\C^n,
\end{multline*}
where $L$ is the Levi form of $-d_D.$
\end{corollary}

\begin{proof} Corollary \ref{c1} implies that it is enough to prove Corollary
\ref{c2} for $M_D=K_D$ which is exactly \cite[Theorem 2]{NT}.
\end{proof}

Corollary \ref{c2} remains true in the $C^{2}$-smooth case, replacing the term
$Cd_D(z)^{\eps_k}$ by any positive number.
\smallskip

Finally, we claim the following weak localization for $M_D$
which can be easily derived from the proof of Theorem \ref{thm}.

\begin{proposition}\label{weak}
Let $p$ be a psh antipeak boundary point of a domain $D$ in $\C^n.$\footnote
{In fact, we need a weaker assumption on the respective antipeak function $\vp:$
$\ds\limsup_{z\to p}\vp(z)<\inf_{D\setminus U}\vp$
for any neighborhood $U$ of $p.$} Then for any neighborhood $U$ of $p$ there
are a neighborhood $V\subset U$ of $p$ and a constant $c>0$ such that
$$M_D(z;X)\ge c M_{D\cap U}(z;X),\ \ \sigma_{D\cap U}(z)\ge c \sigma_D(z),\quad z\in D,\ X\in\C^n.$$

In particular, if $D$ is bounded, this holds for any $p\in\partial D$ and any $V\Subset U.$
\end{proposition}

Note that Proposition \ref{weak} for $A_D$ is claimed in \cite[Remark, pp.~70--71]{Nik}.
\smallskip

\noindent{\bf Acknowledgement.} The authors would like to thank the referee for his/her suggestion
to consider a localization principle of invariant metrics at infinity.

\end{document}